\documentclass[reqno,11pt]{amsart}

\usepackage{mathrsfs}
\usepackage{amssymb,amscd,amsthm, amsfonts,lipsum,amsmath}
\usepackage{times}
\usepackage{flafter}
\usepackage{graphicx}
\usepackage{mathtools}
\usepackage{etoolbox}
\usepackage{geometry}
\usepackage{enumerate}
\usepackage{amscd}
\usepackage{amsthm}
\usepackage{amsmath, nccmath}
\usepackage{
	amsmath,  
	amssymb,  
	amsthm,   
	tikz,     
	fullpage, 
	ytableau, 
	thmtools, 
	hyperref, 
	cite,     
	url       
}
\usepackage{mathtools}
\newtheorem{definition}{Definition}[section]
\newtheorem{example}{Example}[section]
\newtheorem{remark}{Remark}[section]
\newtheorem{corollary}{Corollary}[section]
\newtheorem{theorem}{Theorem}[section]
\newtheorem{proposition}{Proposition}[section]
\newtheorem{lemma}{Lemma}[section]

\def\multiset#1#2{\ensuremath{\left(\kern-.3em\left(\genfrac{}{}{0pt}{}{#1}{#2}\right)\kern-.3em\right)}}

\numberwithin{equation}{section}

\begin{document}
	\title[ On an integral representation of the normalized trace of the $k$-th symmetric power of  matrices and some applications ]{
 On an integral representation of the normalized trace of the $k$-th symmetric tensor power of  matrices and some applications}

\author[H. Issa]{Hassan Issa$^{*}$}
\author[H. Abbass]{Hassan Abbas$^*$}
\author[B. Mourad]{Bassam Mourad$^*$}
\address{*Department of Mathematics, Faculty of Science I,
	Lebanese University, Beirut, Lebanon}
\email[corresponding author]{hissa@uni-math.gwdg.de}
			\begin{abstract}
		Let $A$ be an $n\times n$  matrix and let $\vee^k A$ be its $k$-th symmetric tensor product.   We express the normalized trace of $\vee^k A$ as an integral of the $k$-th powers of the numerical values of $A$ over the unit sphere $\mathbb{S}^{n}$ of $\mathbb{C}^{n}$ with respect to the normalized Euclidean surface measure. Equivalently, this expression in turn can be interpreted as an integral representation for the (normalized) complete symmetric polynomials over $\mathbb{C}^n$. As applications,  we present  a new proof for the MacMahon Master Theorem in enumerative combinatorics. Then, our next application deals with  a
 generalization of  the work of Cuttler et al. in \cite{cuttler} concerning the monotonicity of products of complete symmetric polynomials. In the process, we give  a solution to an open problem that was raised by I. Rovenţa and L. E. Temereanca in \cite{roventa}.
	\end{abstract}
\maketitle
	\bigskip
	\section{Introduction}
	Given $A$  in the space of all $n\times n$ complex matrices $M_n(\mathbb{C})$, we denote its tace  by $tr(A)$  and we shall write $Tr(A):=\frac{tr(A)}{n}$ for the normalized trace. It is well known (\cite{hol,kania}) that $Tr(A)$ can be obtained via the following functional representation
	\begin{equation}\label{dr1}
	\int_{\mathbb{S}^{n}}<A\xi,\xi>d\sigma(\xi)=Tr(A),
	\end{equation}
	where $\mathbb{S}^{n}$ is the unit sphere of $\mathbb{C}^{n}$, $\sigma$ is the probability Lebesgue measure on $\mathbb{S}^{n}$ and $<\cdot,\cdot>$ is the standard inner product on $\mathbb{C}^{n}$. The above formula states that the normalized  trace of a matrix is the expectation of the numerical values of the matrix computed with respect to $d\sigma$. Noting the linearity on both sides of (\ref{dr1}), one obtains the above formula by checking it on matrix bases. When linearity is lost due to some perturbation of the numerical values, new techniques are needed to obtain a reasonable formula. Our main objective inthis paper is to find  the expectation of natural powers of the numerical values computed with respect to $d\sigma$.  More precisely, given two natural numbers $n,k\in\mathbb{N}$ we aim to  find an expression for
	\begin{equation}\label{dr2}
	\int_{\mathbb{S}^{n}}\big(<A\xi,\xi>\big)^kd\sigma(\xi)
	\end{equation}
The preceding integral is of special interest notably in  the analysis of some reproducing kernel Hilbert spaces of analytic functions. More precisely, when addressing  the analysis of the Hardy space $H^2(\mathbb{S}^{n})$ over the unit ball $B(0,1)$ of $\mathbb{C}^{n}$, we recall that the Szeg\"{o} projection $P_S$ is the orthogonal projection of $L^2(\mathbb{S}^{n})$ onto $H^2(\mathbb{S}^{n})$ and is given by
	\begin{equation}\label{dr3}
	P_Sf(z)=\int_{\mathbb{S}^{n}}\frac{f(\xi)}{(1-<v,\xi>)^n}d\sigma(\xi), \quad\mbox{for } v\in B(0,1).
	\end{equation}
	This integral representation for $P_S$ is due to a classical series representation for the reproducing kernel in terms of an orthonormal basis (cf. \cite{saitoh}). Note that for a fixed $v\in B(0,1)$, if $\delta_v$ is the pointwise evaluation at $v$ then the  linear functional $\delta_v\circ P_S$ is  expressed as an infinite series of integral (scalar valued) operators of the form $$c_{n,k}\int_{\mathbb{S}^{n}}(<v,\xi>)^kf(\xi)d\sigma(\xi),$$ where  $c_{n,k}=\binom{n+k-1}{k}$. Thus, it is natural to refer to such operators as the "building blocks" for the Szeg\"{o} projection. In order to clarify the connection of such analysis to our work, we consider a collection of integral transforms $\{T_{A,k}\}_{_{A\in M_n(\mathbb{C}), k\in\mathbb{N}}}$ on $L^2(\mathbb{S}^{n})$ defined by
	\begin{equation}\label{dr4}
	T_{A,k}f(\mu)=\int_{\mathbb{S}^{n}}(<A\mu,\xi>)^kf(\xi)d\sigma(\xi), \quad\mbox{for } f\in L^2(\mathbb{S}^{n}), \quad\mu\in\mathbb{S}^{n}.
	\end{equation}
	Let $e_1=(1,0,\cdots,0)\in\mathbb{C}^n$,  then $\{T_{A,k}f(e_1)\}_{_{A\in M_n(\mathbb{C})}}$ provides a generalization for the "building blocks" of $P_S$. Indeed, given a non-zero $v\in B(0,1)$ consider $A=\lVert v\rVert_2 U$ where $\lVert \cdot\rVert_2$ is the Euclidean norm and  $U$ is a unitary matrix satisfying $Ue_1=\frac{v}{\lVert v\rVert_2}$,  then $$T_{A,k}f(e_1)=\int_{\mathbb{S}^{n}}(<v,\xi>)^kf(\xi)d\sigma(\xi).$$
	It is easy to see that $T_{A,k}$ is a finite rank operator and apparently  the permutation in the $\mu$-coordinate might not allow us to get a canonical decomposition of $T_{A,k}$ for general $A$. This in turn leads to difficulties in  computing the trace of $T_{A,k}$. Alternatively,  as the  kernel $(<A\mu,\xi>)^k$ is continuous in the $\mu$ and $\xi$ variables then the trace of $T_{A,k}$ is given by (\ref{dr2}).
A similar argument holds for the Hilbert-Schmidt norm $\lVert\cdot\rVert_{(2)}$ of $T_{A,k}$ which in support of our line of investigation,  turns out to satisfy the following identity;
	\begin{equation}\label{dr5}
	c^\frac{1}{2k}_{n,k}\lVert T_{A,k}\rVert^\frac{1}{k}_{(2)}=\left(\int_{\mathbb{S}^{n}}\lVert A\xi\rVert_2^{2k}d\sigma(\xi)\right)^\frac{1}{2k}.
	\end{equation}
In this paper, we shall use matrix analysis techniques to find the integral (\ref{dr2}) and then we apply our results to enumerative combinatorics. 
More explicitly, our main objective is to prove the following main result.	
	\begin{theorem}
	Let $A\in M_n(\mathbb{C})$ and let $k\in\mathbb{N}$. Denote by $Tr(\vee^kA)=\frac{tr(\vee^kA)}{c_{n,k}}$ the normalized trace of the k-th symmetric power of A then
	\begin{equation}
	\int_{\mathbb{S}^{n}}\big(<A\xi,\xi>\big)^kd\sigma(\xi)=Tr(\vee^kA).
	\end{equation}
\end{theorem}
Since $Tr(\vee^kA)$ is simply the evaluation of the normalized complete $k$-homogenous polynomials $H_k$ at the eigenvalues of $A$, the preceding integral can be then interpreted as a representation for such polynomials. A direct consequence of the preceding equation is the positivity of $H_{2k}$ on $\mathbb{R}^n$ which has been proved by D. B. Hunter in \cite{hunter}.  A simpler proof to this  fact  that is  based on Schur convexity of the complete symmetric polynomials, was obtained by I. Rovena and L. E. Temereanca in \cite{roventa}. However, we believe that in the context of our work, we offer the simplest proof among the previously mentioned work.

In connection with the analysis of symmetric polynomials, we recall that the Harish-Chandra-Itzykson-Zuber integral which was used by S. Sra in  \cite{sra1} to express the normalized Schur polynomials. Via such a representation together with a technique of Schur's convexity, the author in \cite{sra1} presented  the sufficiency for a conjecture that was proposed in \cite{cuttler} on the monotonicity (via a lexicographic order) for (a product of) normalized Schur polynomials. More explicitly, in \cite{cuttler} the authors presented many inequalities for symmetric functions and in particular, they proved the monotonicity for (a product of) normalized  complete symmetric polynomials on $\mathbb{R}^n_+$. Another consequence of our results lies in exploiting   our integral representation together with a result due Y. L. Tong \cite{tong} on  Schur convexity  to prove 
 that this monotonicity   holds also for  (a product of) normalized  complete symmetric polynomials of even degrees on $\mathbb{R}^n$ not only on $\mathbb{R}^n_+$.

	The paper is organized as follows. Section 2 is devoted to proving our main theorem. In Section 3  we provide a new proof for Macmahon Master Theorem and we obtain an integral representation for a particular determinant. In the last section,   we present  a generalization to the   monotonicity result  in \cite{cuttler} and in the process,  we give a solution to an open problem that was raised  in \cite{roventa}.
	
	\section{Integral formula for the trace of symmetric powers}
	There is a natural way  to provide an exact formula for (\ref{dr2}) by integrating polynomials over the unit sphere of arbitrary dimensional complex spaces. For this reason, it is more convenient to use the standard notation from the theory of complex analysis of several variables. For a multi-index  $\alpha=(\alpha_1,\alpha_2,\cdots,\alpha_n)\in\mathbb N^n_0$ of non-negative integers and for $z\in \mathbb C^n$, we shall write
	$$\overline{z}=(\overline{z_1},\overline{z_2},\cdots,\overline{z_n}),\quad z^{\alpha}=z_1^{\alpha_1}.z_2^{\alpha_2}\cdots z_n^{\alpha_n}, \quad \lvert\alpha\rvert=\alpha_1+\alpha_2+\cdots+\alpha_n, \quad\mbox{and}\quad \alpha!=\alpha_1!\alpha_2!\cdots\alpha_n!.$$
	
	In connection to matrix analysis, given $A=(a_{ij})\in M_n(\mathbb{C})$ we denote by $a:=vec(A^T)$ the vectorization of its transpose  i.e. $a$ is the vector obtained by stacking the rows of A on top of each other (cf. \cite{harver,hndrsn}) so that
	$$a=(a_{11},a_{12},\cdots,a_{1n},a_{21},a_{22}\cdots,a_{2n},\cdots,a_{n1},\cdots,a_{nn})^T\in\mathbb{C}^{n^2}.$$
	Whenever $\alpha\in M_{n^2}(\mathbb{N}_0)$, we will not distinguish between its matrix form and its representation $vec(\alpha^T)\in \mathbb{N}_0^{n^2}$. For example, given $A=(a_{ij})\in M_n(\mathbb{C})$ and $\alpha=(\alpha_{ij})\in M_{n^2}(\mathbb{N}_0)$ we  write $|\alpha|=\sum_{ij} \alpha_{ij}$ and
	$$a^\alpha=\prod_{i,j}^{n^2}a^{\alpha_{ij}}_{ij}.$$
	
	Preliminary to our consideration of the integral formula, we shall express the integrand in (\ref{dr2})  using the above notation.
	\begin{proposition}\label{pdr1}
		Let $A\in M_n(\mathbb{C})$ and $k\in\mathbb{N}$. Given $z\in\mathbb{C}^n$, the following holds:
		\begin{equation}\label{e1}\Big(<Az,z>\Big)^k=\sum_{\substack{|\alpha|=k \\ \alpha\in M_{n^2}(\mathbb{N}_0)}}\binom{k}{\alpha}a^\alpha z^{\big(\sum_i\alpha_{i1},\sum_i\alpha_{i2},\cdots,\sum_i\alpha_{in}\big)} \  \overline{z}^{\big(\sum_j\alpha_{1j},\sum_j\alpha_{2j},\cdots,\sum_j\alpha_{nj}\big)},\end{equation} 
		where $\binom{k}{\alpha}=\frac{k!}{\alpha !}.$
	\end{proposition}
	\begin{proof}
		Applying the multinomial theorem to $\Big(<Az,z>\Big)^k=\Bigg(\sum_{i,j}a_{ij}\overline{z_i}z_j\Bigg)^k$ leads to:
		
		\begin{align*}
		\Big(<Az,z>\Big)^k&=\sum_{\substack{|\alpha|=k \\ \alpha\in\mathbb{N}_0^{n^2}}}\binom{k}{\alpha}\prod_{i,j}^{n^2}a^{\alpha_{ij}}_{ij} \ \overline{z_i}^{\alpha_{ij}}z_j^{\alpha_{ij}}=\sum_{\substack{|\alpha|=k \\ \alpha\in\mathbb{N}_0^{n^2}}}\binom{k}{\alpha}\Bigg[\prod_{i,j}^{n^2}a^{\alpha_{ij}}_{ij}\Bigg]\prod_{i,j}^{n^2}\overline{z_i}^{\alpha_{ij}}z_j^{\alpha_{ij}}\\
		&=\sum_{\substack{|\alpha|=k \\ \alpha\in\mathbb{N}_0^{n^2}}}\binom{k}{\alpha}a^\alpha\prod_{j=1}^n\Bigg[\prod_{i=1}^n\overline{z_i}^{\alpha_{ij}}z_j^{\alpha_{ij}}\Bigg]=\sum_{\substack{|\alpha|=k \\ \alpha\in\mathbb{N}_0^{n^2}}}\binom{k}{\alpha}a^\alpha\prod_{j=1}^nz_j^{\sum_i\alpha_{ij}} \ \overline{z}^{\big(\alpha_{1j},\alpha_{2j},\cdots,\alpha_{nj}\big)}\\
		&=\sum_{\substack{|\alpha|=k \\ \alpha\in M_{n^2}(\mathbb{N}_0) }}\binom{k}{\alpha}a^\alpha z^{\big(\sum_i\alpha_{i1},\sum_i\alpha_{i2},\cdots,\sum_i\alpha_{in}\big)} \ \overline{z}^{\big(\sum_j\alpha_{1j},\sum_j\alpha_{2j},\cdots,\sum_j\alpha_{nj}\big)}.
		\end{align*}
	\end{proof}
	\begin{remark}\label{rdr1}
		Suppose that  $A$ is an  upper triangular matrix then the sum on the right hand side of (\ref{e1}) reduces to those   $\alpha\in M_{n^2}(\mathbb{N}_0)$ satisfying  $|\alpha|=k$ and $\alpha$ is only upper triangular. Indeed,  if  $\alpha_{i_0j_0}\neq0$ for certain $i_0>j_0$ then as $A$ is upper triangular we have $a^\alpha=0$.
	\end{remark}
	Due to the orthogonality of the monomials on $L^2(\mathbb{S}^{n})$, a natural condition is then employed on the row sums as well as the column sums of $\alpha\in M_{n^2}(\mathbb{N}_0)$ when integrating (\ref{e1}). The next lemma examines the type of $\alpha$ under  such  condition.
	\begin{lemma}\label{ldr1}
		Let 	$\alpha\in M_{n^2}(\mathbb{N}_0)$ be an upper triangular matrix  satisfying the condition:
		\begin{equation}\label{dr11}
		\sum_{i=1}^n\alpha_{il}=\sum_{j=1}^n\alpha_{lj}\ \  \mbox{ for every  }  l=1,2\cdots,n,
		\end{equation}
		then $\alpha$ is a diagonal matrix.
	\end{lemma}	
	\begin{proof}
		We shall proceed by induction. For $l=1$, we have $\alpha_{11}=\alpha_{11}+\sum_{j=2}^n\alpha_{1j}$ so that $\alpha_{1j}=0$ for all $j=2,\cdots n$. Assume that for all $k\leq l_0$ we have $\alpha_{kj}=0$ whenever $j\neq k$. By (\ref{dr11}) the $(l_0+1)$ column sum of $\alpha$ satisfies
		
		$$\sum_{i=1}^n\alpha_{i,l_0+1}=\sum_{j=1}^n\alpha_{l_0+1,j}$$
		Notice that, 	$$\sum_{i=1}^n\alpha_{i,l_0+1}=\sum_{i\leq l_0}^n\underbrace{\alpha_{i,l_0+1}}_{=0} +\alpha_{l_0+1,l_0+1}+ \sum_{i> l_0+1}^n\underbrace{\alpha_{i,l_0+1}}_{=0}$$
		The first sum on the right hand side of the above equation vanishes by making se of the induction hypothesis and the second sum vanishes as $\alpha$ is upper triangular.  Similarly, by using  the fact that $\alpha$ is upper triangular we get
		$$\sum_{j=1}^n\alpha_{l_0+1,j}=\sum_{j\leq l_0}^n\underbrace{\alpha_{l_0+1,j}}_{=0} +\alpha_{l_0+1,l_0+1}+ \sum_{j> l_0+1}^n\alpha_{l_0+1,j}.$$
		We conclude that
		$$\alpha_{l_0+1,l_0+1}=\alpha_{l_0+1,l_0+1}+ \sum_{j> l_0+1}^n\alpha_{l_0+1,j},$$
		which  in turn yields that $\alpha_{l_0+1,j}=0$ for all $j> l_0+1$. Thus, $\alpha$ is  diagonal and the proof is complete.
	\end{proof}
	Now, for our purposes we need  a well-known result for the exact value for integration of monomials on spheres.  Such explicit value plays an important role in real, complex and harmonic analysis (cf. for example \cite{axler,folland,krantz}). In particular, for the study of (commuting) Toeplitz operators on the Segal-Bargmann space \cite{bauer,Issa,Issa1}. A simple proof using complex analysis for the next lemma  could be found in \cite{foll} and is due to V. Bargmann and E. Nelson as indicated by G.B. Folland in \cite{foll} (see also Proposition 1.4.9 in \cite{rudin}).
	\begin{lemma}\label{ldr2}
		Let $\sigma$ be the normalized probability measure on the unit sphere $\mathbb{S}^n$  of $\mathbb{C}^n$. For any $\alpha,\beta\in\mathbb{N}^n_0$ we have
		
		\begin{equation}\label{1e16}
		\int_{\mathbb{S}^{n}}\xi^\alpha\overline{\xi}^\beta d\sigma(\xi)=\begin{cases}
		\dfrac{(n-1)!\alpha!}{(n-1+\lvert \alpha\rvert)!} \quad& \alpha=\beta;\\
		0 \quad& \alpha\neq \beta.
		\end{cases}
		\end{equation}
	\end{lemma}
	In order to prove our main theorem, we shall follow the notation used in \cite{bhatia}  for the symmetric tensor product (or power)  of matrices. Recall that, for $k\in\mathbb{N}$ the symmetric tensor product of $\mathbb{C}^n$, denoted by $\vee^k\mathbb{C}^n$ is the subspace of $\otimes^k\mathbb{C}^n$ spanned by all elementary symmetric tensor products $x_1\vee x_2\vee\cdots\vee x_k$ of vectors $x_i\in\mathbb{C}^n$. Given $A\in M_n(\mathbb{C})$, we denote by $\vee^kA$ the $k$th-symmetric tensor power of $A$ i.e. $\vee^kA$ is the matrix in $M_{c_{n,k}}(\mathbb{C})$ representing the linear operator $\vee^kA:\vee^k\mathbb{C}^n\longrightarrow\vee^k\mathbb{C}^n$ whose first definition  on the elementary symmetric tensor product of vectors is given by
	\begin{equation}\label{dr13}
	\vee^kA(x_1\vee x_2\vee\cdots\vee x_k):=Ax_1\vee Ax_2\vee\cdots\vee Ax_k.
	\end{equation}
Recall that the k-th complete elementary symmetric polynomial with n-complex variables is a homogeneous polynomial of order k  given by
$$h_k(z)=\sum_{\substack{|\alpha|=k \\ \alpha\in \mathbb{N}^n_0 }}z^\alpha, \quad z\in\mathbb{C}^n.$$
Here, and in the sequel, if the spectrum $sp(A)=\{\lambda_i\mid i=1,\cdots,n\}$ we shall denote the (un-ordered) n-tuple by $\lambda=\lambda(A):=(\lambda_1,\lambda_2,\cdots,\lambda_n)\in\mathbb{C}^n$ counted with their  multiplicities. Using (\ref{dr13}) it is easy to check that $sp(\vee^kA)=\{\lambda_{i_1}\lambda_{i_2}\cdots\lambda_{i_k}\mid 1\leq i_1\leq i_2\leq\cdots\leq i_k\leq n\}$. This leads to the well-known formula for the trace of $\vee^kA$ in terms of $\lambda_1,\lambda_2,\cdots,\lambda_n$:
		\begin{equation}\label{dr14}
	tr(\vee^kA)=\sum_{1\leq i_1\leq i_2\leq\cdots\leq i_k\leq n}\lambda_{i_1}\lambda_{i_2}\cdots\lambda_{i_k}=h_k(\lambda),
	\end{equation}
Now we are in a position to prove our main result which is Theorem 1.1. For the sake of the reader, we shall restate it again here as follows.
	\begin{theorem}\label{tdr1}
		Let $A\in M_n(\mathbb{C})$ and let $k\in\mathbb{N}$. Denote by $Tr(\vee^kA)=\frac{tr(\vee^kA)}{c_{n,k}}$ the normalized trace of the k-th symmetric power of A,  then
		\begin{equation}\label{dr15}
		\int_{\mathbb{S}^{n}}\big(<A\xi,\xi>\big)^kd\sigma(\xi)=Tr(\vee^kA).
		\end{equation}
	\end{theorem}
	\begin{proof}
		Using Schur's decomposition, one can write $A=U^\star TU$ with $T$ upper triangular and $U$ is unitary.
		So that the change of variable $u=U\xi$ provides
		$$\int_{\mathbb{S}^{n}}\big(<A\xi,\xi>\big)^kd\sigma(\xi)=\int_{\mathbb{S}^{n}}\big(<TU\xi,U\xi>\big)^kd\sigma(U^\star u)=\int_{\mathbb{S}^{n}}\big(<Tu,u>\big)^kd\sigma(u),$$
		where the last equality follows from the invariance of $\sigma$  under unitary transformation. Thus, we may assume that $A$ is upper triangular and $\lambda=(a_{11},a_{22},\cdots,a_{nn})$.
		By Proposition \ref{pdr1} and Remark \ref{rdr1}, we obtain that
		$$\int_{\mathbb{S}^{n}}\big(<A\xi,\xi>\big)^kd\sigma(\xi)$$
		is given by $$\sum_{\substack{|\alpha|=k \\ \alpha\in M_{n^2}(\mathbb{N}_0) \\\alpha \text{ upper triangular} }}\binom{k}{\alpha}a^\alpha\int_{\mathbb{S}^{n}} \xi^{\big(\sum_i\alpha_{i1},\sum_i\alpha_{i2},\cdots,\sum_i\alpha_{in}\big)} \ \overline{\xi}^{\big(\sum_j\alpha_{1j},\sum_j\alpha_{2j},\cdots,\sum_j\alpha_{nj}\big)}d\sigma(\xi).$$
		Following Lemma \ref{ldr2}, the preceding integral terms vanishes for all $\alpha$ except when 	$$\sum_{i=1}^n\alpha_{il}=\sum_{j=1}^n\alpha_{lj}, \ \ \ \mbox{ for every  }  l=1,2\cdots,n.$$
		By Lemma \ref{ldr1}, $\alpha$ is a diagonal matrix so that
		$$a^\alpha= \lambda_1^{\alpha_{11}}\lambda_2^{\alpha_{22}}\cdots\lambda_2^{\alpha_{nn}}=\lambda^\alpha,$$
		where in the last equality we considered   $\alpha=(\alpha_1,\alpha_2,\cdots,\alpha_n)=(\alpha_{11},\alpha_{22},\cdots,\alpha_{nn})$ as an $n$-tuple in $\mathbb{N}^n_0$. This reduces the preceding integral to the following form
		
		$$\int_{\mathbb{S}^{n}}\big(<A\xi,\xi>\big)^kd\sigma(u)=\sum_{\substack{|\alpha|=k \\ \alpha\in \mathbb{N}^n }}\binom{k}{\alpha}\lambda^\alpha\int_{\mathbb{S}^{n}} \xi^{\big(\alpha_{11},\alpha_{22},\cdots,\alpha_{nn}\big)} \ \overline{\xi}^{(\alpha_{11},\alpha_{22},\cdots,\alpha_{nn}\big)}d\sigma(\xi).$$
		Another application of Lemma \ref{ldr2} for the norms of monomials, yields
		\begin{align}
		\int_{\mathbb{S}^{n}}\big(<A\xi,\xi>\big)^kd\sigma(\xi)&=\sum_{\substack{|\alpha|=k \\ \alpha\in \mathbb{N}^n_0 }}\binom{k}{\alpha}\lambda^\alpha\dfrac{(n-1)!\alpha!}{(n-1+\lvert \alpha\rvert)!}\notag\\
		&=\sum_{\substack{|\alpha|=k \\ \alpha\in \mathbb{N}^n_0 }}\frac{k!}{\alpha !}	\lambda^\alpha\dfrac{(n-1)!\alpha!}{(n-1+k)!}\notag\\
		&=\frac{k!(n-1)!}{(n-1+k)!}\sum_{\substack{|\alpha|=k \\ \alpha\in \mathbb{N}^n_0}} 	\lambda^\alpha\label{pq}\\
		&=\frac{tr(\vee^kA)}{c_{n,k}}.\notag
		\end{align}

	\end{proof}
	Formula (\ref{dr15}) merits (at least a quick) discussion on its important role in the analysis of function spaces. It is trivial that (\ref{dr15}) generalizes (\ref{dr1}) from $k=1$ to arbitrary $k\in\mathbb{N}$. With this in mind, a direct application of ((\ref{dr1})) to $\vee^kA\in M_{c_{n,k}}(\mathbb{C})$ together with (\ref{dr15}) provides
	\begin{equation}\label{dr16}
	\int_{\mathbb{S}^{c_{n,k}}}<\vee^kA\mu,\mu>d\sigma(\mu)=\int_{\mathbb{S}^{n}}\big(<A\xi,\xi>\big)^kd\sigma(\xi).
	\end{equation}
	The above formula can be viewed as a reduction from integrals over $\mathbb{S}^{c_{n,k}}$ to integrals over $\mathbb{S}^n$. This leads to various applications in the study of the analysis of reproducing kernel Hilbert spaces of analytic function over specific domains. For example, the left hand side of (\ref{dr16}) is the Szeg\"{o} projection onto $H^2(\mathbb{S}^{c_{n,k}})$ of $$\mathbb{S}^{c_{n,k}}\ni\mu\longrightarrow<\vee^kA\mu,\mu>$$at $z=0\in\mathbb{C}^{c_{n,k}}$,  while the right hand side  is the Szeg\"{o} projection onto $H^2(\mathbb{S}^n)$ of $$\mathbb{S}^{n}\ni\xi\longrightarrow \big(<A\xi,\xi>\big)^k$$ at $z=0\in\mathbb{C}^n.$
Another application is connected to the study of the composition  of  Toeplitz operators on the Segal-Bargmann space. We refer the reader to \cite{agbor,bauer1,coburn,Issa1} where integral reduction was frequently used to obtain a composition formula for some classes of Toeplitz operators.
	
	
	Notice that Formula (\ref{dr15})  shows that the integral  depends on the spectrum of the matrix.  However,  by expressing the complete symmetric polynomials  as  power sums we obtain that the integral depends  only on moments of the eigenvalues.  For this reason, we present the following well-known result on representing the complete symmetric polynomials in terms of power sums.
	\begin{lemma}[\cite{mcd} pp. 24-25]\label{ldr3}
		Let $k\in\mathbb{N}$ and let $h_k(x)=\sum_{\substack{|\alpha|=k \\ \alpha\in \mathbb{N}^n }}x^{\alpha}$. For each $r\in\mathbb N$ put $p_r(x)=\sum_{i=1}^n x_i^r.$ Given a non-zero compactly supported sequence arranged in decreasing  order $\beta=(\beta_1,\beta_2,\cdots,\beta_l,0,0,0\cdots)\in\mathbb{N}^\infty$, define
		$$p_\beta(x)=p_{\beta_1}(x)p_{\beta_2}(x)\cdots p_{\beta_l}(x),$$
		with $\beta_l$ is the leading entry of $\beta$. Then,
		\begin{equation}\label{dr17}
		h_k(x)=\sum_{|\beta|=k}z^{-1}_\beta p_\beta(x),\end{equation}
		where  \begin{equation}\label{ldr18}
		z_\beta=\prod_{i\geq1} i^{m_i}m_i!
		\end{equation} and $m_i$ is the number for which $i$ is occurring in $\beta$.
	\end{lemma}
	We shall apply the preceding lemma to express the trace of the symmetric tensor product of a matrix $A$ in-terms of product of traces of powers of $A$.  Following the notation in the above lemma, Formula (\ref{dr15}) can then be reformulated in-terms of the $rth$-moments of the eigenvalues $p_r(\lambda)=\sum_{i=1}^n\lambda_i^r$ as follows:
	\begin{equation}\label{dr19}
	c_{n,k}\int_{\mathbb{S}^{n}}\big(<A\xi,\xi>\big)^kd\sigma(\xi)=\sum_{\substack{|\beta|=k \\ \beta_1\geq\beta_2\geq\cdots\geq\beta_k\geq0}}z^{-1}_\beta p_\beta(\lambda)=\sum_{\substack{|\beta|=k \\ \beta_1\geq\beta_2\geq\cdots\geq\beta_k\geq0}}z^{-1}_\beta\prod_{t=1}^l tr(A^{\beta_t}).
	\end{equation}
	Notice that the condition for which $\beta\in \mathbb{N}^k_0$ is trivial since $\beta$ is decreasing with $|\beta|=k$. Using the basic fact that the product of traces is the trace of tensor products, Eq. (\ref{dr19}) can be rewritten as follows
	\begin{equation}\label{dr20}
	c_{n,k}\int_{\mathbb{S}^{n}}\big(<A\xi,\xi>\big)^kd\sigma(\xi)=\sum_{\substack{|\beta|=k \\ \beta_1\geq\beta_2\geq\cdots\geq\beta_k\geq0}}z^{-1}_\beta tr\Big[\bigotimes_{t=1}^lA^{\beta_t}\Big],
	\end{equation}
	where $\bigotimes$ denotes the tensor (Kronecker) product of matrices.
	
	With an inspection of equations (\ref{dr19}) and (\ref{dr20}) one obtains an "exhaustive behavior" of the $l$-index of $\beta$. Indeed, in both equations there is no-need to worry about $\beta_i$ whenever $i>l$ so fixing $l$ will be helpful in computation. Moreover,  regarding (\ref{dr20}) one might be interested in summing up tensor powers of matrices (of same dimension). This also amounts to collecting the multi-indices having the same index in the  leading term. For this reason, we consider
	$$S^k:=\{\beta=(\beta_1,\beta_2,\cdots,\beta_k)\in\mathbb{N}_0^k\mid \beta_1\geq\beta_2\geq\cdots\geq\beta_k \mbox{ and } |\beta|=k\}$$
	and for each $l=1,2,\cdots,k$ we set $$S^k_l=\{\beta\in S^k\mid \beta_l\neq0 \mbox{ and } \beta_i=0  \mbox{ for all } i>l\}.$$
	It is worthy to mention here that $\{S^k_l\}_{l=1,2,\cdots,k}$ forms a partition for $S^k$ and thus we conclude from the previous discussion the following new formulation of Theorem \ref{tdr1}.
	\begin{corollary}\label{cdr1}
		Let $A\in M_n(\mathbb{C})$ and let $k\in\mathbb{N}$ then the following holds
		\begin{align}
		c_{n,k}\int_{\mathbb{S}^{n}}\big(<A\xi,\xi>\big)^kd\sigma(\xi)&=\sum_{l=1}^k\sum_{\beta\in S^k_l}\frac{1}{z_\beta}\prod_{t=1}^l tr(A^{\beta_t})\label{dr21}\\
		&= \sum_{l=1}^ktr\left[ \Big(\sum_{\beta\in S^k_l} \frac{1}{z_\beta}\bigotimes_{t=1}^lA^{\beta_t}\Big)\right]\label{dr22}.
		\end{align}
	\end{corollary}
	The next example is devoted to finding the explicit value of the integral formula for the cases $k=2,\cdots 6$.  Since  going from (\ref{dr21}) to (\ref{dr22}) is  direct, we shall only apply (\ref{dr22})  for the cases $k=2$ and $k=3$.
	\begin{example}\label{edr1}
		For a given	 $k,n\in\mathbb{N}$ put $c_{n,k}=\binom{n+k-1}{k}$ then for any  $A\in M_n(\mathbb{C})$ the following formulas hold
		\begin{fleqn}[\parindent]
			\begin{equation}\label{ed14}
			\begin{split}c_{n,2}\int_{\mathbb{S}^{n}}\big(<A\xi,\xi>\big)^2d\sigma(\xi)=\frac{1}{2}tr(A^2)+\frac{1}{2}tr(A\otimes A)=\frac{1}{2}tr(A^2)+\frac{1}{2}(tr A)^2\end{split}
			\end{equation}
			\begin{equation}\label{ed15}
			\begin{split}
			c_{n,3}\int_{\mathbb{S}^{n}}\big(<A\xi,\xi>\big)^3d\sigma(\xi)&=\frac{1}{3}tr(A^3)+\frac{1}{2}tr(A^2\otimes A)+\frac{1}{6}tr(A\otimes A\otimes A)\\
			&=\frac{1}{3}tr(A^3)+\frac{1}{2}[tr(A^2)]tr(A)+\frac{1}{6}[tr A]^3.
			\end{split}
			\end{equation}
			\begin{equation}\label{ed16}
			c_{n,4}\int_{\mathbb{S}^{n}}\big(<A\xi,\xi>\big)^4d\sigma(\xi)
			=\frac{1}{4}tr(A^4)+\frac{1}{3}tr(A^3)tr(A)+\frac{1}{8}[tr (A^2)]^2+\frac{1}{4}tr (A^2)[tr(A)]^2+\frac{1}{24}[tr (A)]^4.
			\end{equation}
			\begin{equation}\label{ed17}
			\begin{split}c_{n,5}\int_{\mathbb{S}^{n}}\big(<A\xi,\xi>\big)^5d\sigma(\xi)
			=&\frac{1}{5}tr(A^5)+\frac{1}{4}[tr(A^4)]tr(A)+\frac{1}{6}[tr (A^3)][tr (A^2)]+\frac{1}{6}[tr (A^3)][tr (A)]^2\\
			&+\frac{1}{8}[tr (A^2)]^2(tr A)+\frac{1}{12}tr (A^2)[tr (A)]^3+\frac{1}{120}[tr (A)]^5.
			\end{split}
			\end{equation}
			\begin{equation}\label{ed18}
			\begin{split}c_{n,6}\int_{\mathbb{S}^{n}}\big(<A\xi,\xi>\big)^6d\sigma(\xi)=&\frac{1}{6}tr(A^6)+\frac{1}{5}[tr(A^5)]tr(A)+\frac{1}{8}[tr (A^4)][tr (A^2)]+\frac{1}{18}[tr (A^3)]^2\\
			&+\frac{1}{8}[tr (A^4)][tr (A)]^2+\frac{1}{6}[tr (A^3)][tr (A^2)][tr(A)]+\frac{1}{48}[tr (A^2)]^3\\
			&+\frac{1}{18}[tr (A^3)][tr (A)]^3+\frac{1}{16}[tr (A^2)]^2[tr (A)]^2\\
			&+\frac{1}{48}[tr (A^2)][tr (A)]^4+\frac{1}{720}[tr (A)]^6.
			\end{split}
			\end{equation}
		\end{fleqn}
	\end{example}
	We conclude this  section by noting that Eq. (2.15) can be obtained from Corollary 2.3 and Theorem 2.4 in \cite{hol}, while to the best of our knowledge Equations (2.16)-(2.19) are new.
	\section{A new proof of MacMahon Master Theorem}
	This section is devoted to provide a simple proof for MacMahon Master theory. This theory has been central in combinatorics and in the theory of angular momentum of systems of particles as well. The original proof is due to P. A. MacMahon \cite{mcmhn} while another proof depending on complex analysis was provided by I. J. Good in \cite{good}. Our proof seems to be the simplest.
 Making use of the results in the previous section, let us first present an integral representation of  a particular determinant as follows.
	\begin{corollary}\label{cec1}
		Let $B\in M_n(\mathbb{C})$ with $\lVert B\rVert_2=\sup_{\lVert x\rVert=1}\lVert Bx\rVert_2<1$ then
		\begin{equation}\label{ghj}
	\int_{\mathbb{S}^{n}}\dfrac{d\sigma(\xi)}{\big(<(I_n-B)\xi,\xi>\big)^n}=\det (I_n-B)^{-1}.
		\end{equation}
	\end{corollary}
	\begin{proof}
		By Theorem \ref{tdr1}, we know that for each $k\in\mathbb{N}$ the following identity holds
		$$\int_{\mathbb{S}^{n}}  \binom{n+k-1}{k}\big(<B\xi,\xi>\big)^kd\sigma(\xi)=\sum_{\substack{|\alpha|=k \\ \alpha\in \mathbb{N}^n }}\lambda^\alpha.$$
		 Since 	$\lVert B\rVert_2<1$ the series  $\sum_{k=0}^\infty\binom{n+k-1}{k}\big(<B\xi,\xi>\big)^k$ converges absolutely and is dominated by $\dfrac{1}{(1-\lVert B\rVert_2)^n}$.   Hence by Lebesgue dominated convergence theorem we get
		$$\int_{\mathbb{S}^{n}} \sum_{k=0}^\infty\binom{n+k-1}{k}\big(<B\xi,\xi>\big)^k=\sum_{k=0}^\infty\sum_{\substack{|\alpha|=k \\ \alpha\in \mathbb{N}^n }}\lambda^\alpha$$
		i.e. \begin{equation}\label{ggj}\int_{\mathbb{S}^{n}} \big(1-<B\xi,\xi>\big)^{-n}d\sigma(\xi)=\sum_{\alpha\in \mathbb{N}^n}\lambda^\alpha=\prod_{j=1}^{n}\dfrac{1}{1-\lambda_j}=\dfrac{1}{\det(I_n-B)}.\end{equation}
			\end{proof}
	\begin{remark}
Note that 	the monotone convergence theorem ensures that Eq. (\ref{ghj}) remains true for any positive semidefinite matrix B. Also, it is easy to see that for any	$c\in\mathbb{C}$ with $|c|>1$ and for any non-zero matrix  $A\in M_n(\mathbb{C})$,  we have
		$$\int_{\mathbb{S}^{n}}\dfrac{d\sigma(\xi)}{\big(<(c\lVert A\rVert_2I_n-A)\xi,\xi>\big)^n}=\det (c\lVert A\rVert_2I_n-A)^{-1}.$$
		\end{remark}
	
Now, we are ready to present our proof for MacMahon Master theorem.
\begin{theorem}[MacMahon]
Let $A=(a_{ij})\in M_n(\mathbb{C})$, $\alpha\in\mathbb{N}^n$ and $x=(x_1,x_2,\cdots,x_n)^T$ be a formal variable. Then the coefficient of $x^\alpha in the expression of \big(Ax\big)^\alpha$ is equal to the coefficient of $x^\alpha$ in the expansion of $\det(I_n-\Delta(x)B)^{-1}$, where $\Delta(x)$ is the diagonal matrix given by $\Delta(x)=\texttt{diag}(x_1,_2,\cdots,x_n)$.
\end{theorem}
\begin{proof}
 For$k\in\mathbb{N}$ and for each $\beta\in\mathbb{N}^n$ with $|\beta|=k$, denote by $c_\beta$  the coefficient of $x^\beta$ in $\big(Ax\big)^\beta$. Since $\lambda\big(\Delta(x)A\big)=\lambda\big(A\Delta(x)\big)$ then by (\ref{pq}) we have
\begin{align*}
	h_k\Big(\lambda\big(\Delta(x)A\big)\Big)&=c_{n,k}\int_{\mathbb{S}^{n}}\big(<A\Delta(x)\xi,\xi>\big)^kd\sigma(\xi)\\
	 &=c_{n,k}\int_{\mathbb{S}^{n}}\bigg[\Big<\big(\sum_{j=1}^na_{1j}x_j\xi_j,\sum_{j=1}^na_{2j}x_j\xi_j,\cdots,\sum_{j=1}^na_{nj}x_j\xi_j\big)^T,\xi\Big>\bigg]^kd\sigma(\xi)
\end{align*}
Applying again the multinomial theorem, we get
\begin{align}\label{ppopo}
h_k\Big(\lambda\big(\Delta(x)A\big)\Big)&=c_{n,k}\sum_{\substack{|\beta|=k \\ \beta\in \mathbb{N}^n }}\binom{k}{\beta}\int_{\mathbb{S}^{n}}c_\beta x^\beta\xi^\beta\overline{\xi}^\beta d\sigma(\xi)=c_{n,k}\sum_{\substack{|\beta|=k \\ \beta\in \mathbb{N}^n }}\binom{k}{\beta}\dfrac{(n-1)!\beta!}{(n-1+\lvert \beta\rvert)!}\\
&=\dfrac{(n+k-1)!}{(n-1)!k!}	\sum_{\substack{|\beta|=k \\ \beta\in \mathbb{N}^n }}\frac{k!}{\beta!}c_\beta x^\beta\dfrac{(n-1)!\beta!}{(n-1+k)!}=\sum_{\substack{|\beta|=k \\ \beta\in \mathbb{N}^n }}c_\beta x^\beta,
\end{align}
where the second equality in (\ref{ppopo}) follows from (\ref{1e16}). Finally, taking the sum over  $k\in\mathbb{N}$ and using (\ref{ggj}) we obtain
$$\det(I_n-\Delta(x)B)^{-1}=\sum_{\alpha\in \mathbb{N}^n}\lambda^\alpha\Big((\Delta(x)A\Big)=\sum_{k=0}^\infty h_k\Big(\lambda\big(\Delta(x)A\big)\Big)=\sum_{\beta\in \mathbb{N}^n}c_\beta x^\beta.$$
	\end{proof}
	\section{On the monotonicity of complete symmetric polynomials}
	In this section, we extend the definition of k-homogeneous complete symmetric polynomials where $k\in\mathbb{N}$ to arbitrary $p\in\mathbb{R}^+$.  We then exploit the definition to introduce a natural extension for the term-normalized homogeneous polynomial. Using the theory of majorization,  we obtain a Schur convexity of this extension. As a consequence , we obtain a  generalization of Theorem 7.3 in \cite{cuttler} and  we answer an open problem raised in \cite{roventa} as well.

Throughout this section, we shall follow the standard notation used  in the study of symmetric polynomials and the theory of majorization. Thus, for   $m$-tuples $\lambda, \mu\in\mathbb{R}^m_+$ we say that $\lambda$ is majorized by $\mu$ and we will write   $\lambda\preceq\mu$ if $$\sum^k_{i=1} \lambda_{[i]}\leq\sum^j_{i=1} \mu_{[i]}, \  j=1\cdots n-1\quad\mbox{ and }\quad|\lambda|=|\mu|,$$
where $x_{[i]}$ is the $i$he component obtained from $x=(x_1,\cdots,x_n)$ after arranging its components  in decreasing order.
In the case where $\mu\in\mathbb{R}^n_+$ and $n<m$, then   we shall complete the components of $\mu$ to an $m$-tuple by adding zeros so that the preceding notion is still well defined. If $\lambda$ and $\mu$ are partitions of (possibly different) natural integers then we shall write $$\lambda\sqsubseteq\mu\quad \mbox{if}\quad\frac{\lambda}{|\lambda|}\preceq\frac{\mu}{|\mu|},$$
For each $k\in\mathbb{N}$, we shall denote by  $H_k$ to be the normalized complete symmetric polynomial on $\mathbb{C}^n$ i.e.
\begin{equation}\label{xfxf1}
H_k(z)=\frac{h_k(z)}{c_{n,k}}.
\end{equation}
Given $\lambda=(\lambda_1,\cdots,\lambda_m)\in\mathbb{N}^m$, the term-normalized homogeneous complete symmetric function is defined on $\mathbb{C}^n$ by
\begin{equation}\label{xfxf2}
H_\lambda(z)=\prod_{i-1}^mH_{\lambda_i}(z).
\end{equation}
In addition, for $x\in\mathbb{R}^n_+$ and for the map defined by
\begin{equation}\label{xfxf-1}
\mathfrak{H}_\lambda(x):=\sqrt[|\lambda|]{H_\lambda(x)},
\end{equation}
the authors in \cite{cuttler} proved the monotonicity of $\mathfrak{H}_\lambda$ on $\mathbb{R}^n_+$ as follows.
\begin{theorem}[Theorem 7.3,  \cite{cuttler}]
	Given integer partitions $\lambda$ and $\mu$ with $\lambda\sqsubseteq\mu$,  then for any $x\in\mathbb{R}^n_+$ we have
	\begin{equation}\label{xfxf3}
	\mathfrak{H}_\lambda(x)\leq\mathfrak{H}_\mu(x).
	\end{equation}
	\end{theorem}
In \cite{sra1},  the author used the Harish-Chandra-Itzykson-Zuber (HCIZ) integral to represent the term-normalized Schur polynomials in order to obtain a monotonicity result concerning  a conjecture formulated in \cite{cuttler}. As mentioned earlier, we shall next use the integral representation obtained in Section 2 to provide a generalization of the preceding theorem. Our work is based on majorization methods for moments and thus differs from the algebraic approach used in \cite{cuttler}.  Since the integrand in HCIZ involves two matrices, the author  in \cite{sra1} was able to express the term-normalized Schur polynomial as a single integral (over the set of unitary matrices). The existence of the exponential term in HCIZ was needed to obtain a   Schur convexity result. In our case, the integral is over a finite product of spheres and the integrand is given by powers of numerical values and so many difficulties naturally arise.

 For an arbitrary $z\in\mathbb{C}^n$, we shall  write $Z:=\texttt{diag}(z_1,\cdots,z_n)\in M_n(\mathbb{C})$. Thus, by (\ref{dr14}) and (\ref{dr15}) we have the following representation for the normalized complete symmetric polynomials
 \begin{equation}\label{xfxf4}
 H_k(z)=\int_{\mathbb{S}^{n}}\big(<Z\xi,\xi>\big)^kd\sigma(\xi)
 \end{equation}
 In \cite{hunter}, the author proved that $H_{2k}$ is positive definite on $\mathbb{R}^n$ using a differential operator approach and some identities between symmetric polynomials. Following a post by T. Tao concerning this issue,  two proofs were presented by I. Roventa and L. E. Temereanca for the positivity but still based on polynomial identities and some computations. In the next corollary, we  present our own  proof not only for the positivity but also for the definitness of this property for $H_{2k}$.
 \begin{corollary}
 The even degree complete symmetric polynomials are positive definite on $\mathbb{R}^n$.
 \end{corollary}
\begin{proof}
Using the representaion (\ref{xfxf4}),  we know that for any $x\in\mathbb{R}^n$ we have \begin{equation}\label{xfxf5}
H_{2k}(x)=\int_{\mathbb{S}^{n}}\big(<X\xi,\xi>\big)^{2k}d\sigma(\xi),
\end{equation}
where $X=\texttt{diag}(x_1,\cdots,x_n)$ so that  positivity is clear. If $H_{2k}(x)=0$ then the map $$\mathbb{S}^n\ni\xi\longrightarrow\big(<X\xi,\xi>\big)^2\in\mathbb{R}_+$$
is equal to zero almost everywhere on $\mathbb{S}^n$. By continuity of the preceding map,  we obtain $<X\xi,\xi>=0$ for all $\xi\in\mathbb{S}^n$ or equivalently $x=0$.
\end{proof}
Motivated by the latter property, we introduce the following definition and we shall use the same notation as in (\ref{xfxf-1})
\begin{definition}
	Given $\nu=(\nu_1,\nu_2,\cdots,\nu_m)\in \mathbb{N}^m$,  let the map $\mathfrak{H}_{2\nu}$  be defined on $\mathbb{R}^n$  by $$\mathfrak{H}_{2\nu}(x):=\sqrt[2|\nu|]{H_{2\nu}(x)}.$$
\end{definition}

In the following proposition, we examine the possibility of an  extension Formula (\ref{xfxf5}) from  the case of integer powers to the case of arbitrary real powers.
 \begin{proposition}
 Let $X\in\ M_n(\mathbb{C})$ be a non-zero positive semidefinite matrix. Then for any $p\in\mathbb{R}$, the integral
  \begin{equation}\label{xfxf10}
\int_{\mathbb{S}^{n}}\big(<X\xi,\xi>\big)^{p}d\sigma(\xi)
 \end{equation}
 is well defined with values in $]0,\infty[$.
 \end{proposition}
 \begin{proof}
 Suppose first that $p\geq0$, then clearly the proposition follows from the inequality
 $$\int_{\mathbb{S}^{n}}\big(<X\xi,\xi>\big)^{p}d\sigma(\xi)\leq \int_{\mathbb{S}^{n}}\big(w(X)\big)^{p}d\sigma(\xi)=\big(w(X)\big)^{p},$$
 where $w(\cdot)$ denotes the numerical radius norm. On the other hand, if  $p<0$ and if $v$ denotes the Lebesgue measure on $\mathbb{C}^n$. then we  consider the following set
 $$F:=\Big\{z\in\mathbb{C}^n\mid <Xz,z>=0\Big\}.$$
 As $X\geq0$ and $X\neq0$ then $F=\ker X$ which is an intersection of hyperplanes in $\mathbb{C}^n$.  So that $v(F)=0$ or equivalently $F\backslash\{0\}$ is measurable  and  $v(F\backslash\{0\})=0.$ Next, we  consider the set defined by
 $$E:=\Big\{\xi\in\mathbb{S}^n\mid <X\xi,\xi>=0\Big\}.$$
 By the definition of the sigma-algebra  on $\mathbb{S}^n$ (cf. for example Chapter 6, Section 3 in \cite{stein}), the set $E$ is measurable if and only if the set
 $$\widetilde{E}:=\Big\{z\in B(0,1)\mid z\neq0 \mbox{ and } \frac{z}{|z|}\in E\Big\}$$
 is Lebesgue measurable in $\mathbb{C}^n$. Notice that by representing $z$ in polar coordinates $z=r\xi$ with $r\in]0,\infty[$ we obtain
  $$\widetilde{E}=B(0,1)\cap F\backslash\{0\}$$
  and hence $E$ is measurable.  Let $\mu_1$ be the corresponding radial measure on $(0,\infty)$ i.e. $\mu_1(\theta)=\int_\theta r^{2n-1}dr$ for every Lebesgue measurable set $\theta$ in $(0,\infty)$. Since $F\backslash\{0\}=(0,\infty)\times E=\bigcup_{l\in\mathbb N}(0,l)\times E$ then $$0=v(F\backslash\{0\})=\lim_{l\rightarrow\infty}\mu_1((0,l))\sigma(E)=\lim_{l\rightarrow\infty}\frac{l}{2n}\sigma(E)$$ which holds only in the case when $\sigma(E)=0$. Let  $f$ be the map on $\mathbb{S}^n$ defined by  $f(\xi)=\big(<X\xi,\xi>\big)^{p}$ with values in $[0,\infty]$. As $f$ is real valued continuous function on $E^c:=\mathbb{S}^n\backslash\{E\}$ then $f$ is measurable on $E^c$ and therefore on $\mathbb{S}^n$. Moreover, as $E$ is a null-$\sigma$-set then
  \begin{equation}\label{xfxf-2}
  \int_{\mathbb{S}^{n}}\big(<X\xi,\xi>\big)^{p}d\sigma(\xi)=\int_{\mathbb{S}^{n}} \big(<X\xi,\xi>\big)^{p}\chi_{E^c}(\xi)d\sigma(\xi).
  \end{equation}
 By the unitary invariance of $\sigma$, one can assume that $X$ is diagonal and the preceding formula remains true. Indeed, compared to the work in Section 1 the only factor that needs to be examined is $\chi_{E^c}$. However, writing $X=U^\star\Delta U$ with $\Delta$ being diagonal matrix and $U\in U_n$ then by the invariance of the Euclidean norm under unitary transformation we have
 $$E=\Big\{\xi\in\mathbb{S}^n\mid <X\xi,\xi>=0\Big\}=\Big\{\xi\in\mathbb{S}^n\mid <\Delta U\xi,U\xi>=0\Big\}=\Big\{\eta\in\mathbb{S}^n\mid <\Delta \eta,\eta>=0\Big\},$$
where the last equality folllows from th e fact that $U$ is an isometry. The above equation shows that
 $\chi_{E}(U\xi)=1$ if and only if  $\xi\in\mathbb{S}^n$ and $<\Delta U\xi,U\xi>=0$ which is equivalent to say that $\xi\in E$. Therefore, $\chi_{E^c}(U\xi)=\chi_{E^c}(\xi)$ and thus we can assume $X=\texttt{diag}(x_1,\cdots,x_n)$ with $x=\lambda(X)$ being a non-zero vector in $\mathbb R^n_+$. It remains to prove that the value of the integral is strictly positive. For this, we let $\check{x}=\min\{x_i\mid x_i\neq0\}$ and $\hat{x}=\max\{x_i\mid x_i\neq0\}$ then clearly $\check{x}>0$,  $\hat{x}>0$ and
 $$\hat{x}^p\leq\big(<X\xi,\xi>\big)^{p}\leq \check{x}^p,$$
 where the above inequality holds for all $\xi\in E^c$. Therefore, by (\ref{xfxf-2}) we get
  $$\hat{x}^p\leq\int_{\mathbb{S}^{n}}\big(<X\xi,\xi>\big)^{p}d\sigma(\xi)\leq\check{x}^p.$$
 \end{proof}
Motivated by the above result,  we introduce the following generalization for normalized  (and term normalized) complete homogeneous polynomials on $\mathbb{R}^n_+$.
\begin{definition}
Let $p\in\mathbb{R}$. The $p$-homogeneous function $H_p$ is the map defined on $\mathbb{R}^n_+$ by
 \begin{equation}\label{xfxf8}H_p(x)=\begin{cases}
 \int_{\mathbb{S}^{n}}\big(<X\xi,\xi>\big)^{p}d\sigma(\xi)& \mbox{if} \quad x\neq0\\
0, & \mbox{if}\quad x=0,
\end{cases}\end{equation}
where $X=\texttt{diag}(x_1,\cdots,x_n)$. In addition, given $\lambda=(\lambda_1,\cdots,\lambda_m)\in\mathbb{R}^m$ we define the term-normalized  $\lambda$-function by
\begin{equation}\label{xfxf6}
H_\lambda(x)=\prod_{i=1}^mH_{\lambda_i}(x).
\end{equation}
 \end{definition}
In \cite{tong}, and based on Muirhead Theorem,  Y. L. Tong  provided an inequality for the expectation of product of independent random variables with a majorization condition. The case where the probability measure is considered on $[0,\infty[$ can be found in \cite{Marshall} (see p. 107). We formulate this result in the context of our work and we provide a proof motivated by the work of S. Sra in \cite{sra1} which is based on a Schur's result for convex symmetric functions.
\begin{theorem}\label{tmm1}
	Let $f:\mathbb{S}^n\longrightarrow \mathbb{R}_+$ be a measurable function. Let $I\subseteq\mathbb{R}$ be an interval and assume that
	$$\int_{\mathbb{S}^{n}}f^p(\xi)d\sigma(\xi)<\infty$$
	for all $p\in I$. Then for any $m\in\mathbb{N}$ the map \begin{equation}\label{fz1}F(\lambda):=\prod_{i=1}^m\int_{\mathbb{S}^{n}}f^{\lambda_i}(\xi)d\sigma(\xi)\end{equation}
is Schur convex on $I^m$. Moreover, if $\lambda,\mu\in I^m$ are integer partitions  with $\lambda\sqsubseteq\mu$ then
\begin{equation}\label{fz2}
\sqrt[|\lambda|]{F(\lambda)}\leq\sqrt[|\mu|]{F(\mu)}
\end{equation}
\end{theorem}
\begin{proof}
As $F$ is continuous and symmetric it is sufficient to prove $F$ is convex (cf. \cite{Marshall} p. 97). Now, by Fubini's theorem we write
$$F(\lambda)=\int_{\mathbb{S}^{n}\times\mathbb{S}^{n}\times\cdots\mathbb{S}^{n}}\prod_{i=1}^mf^{\lambda^i}d(\xi^i),$$
where $\xi^i$ is the coordinates in the $i$-th copy of $\mathbb{S}^{n}$. Thus, for $\lambda,\mu\in I^m$, the power mean inequality yields
\begin{equation*}
F(\frac{\lambda+\mu}{2})=\int_{\mathbb{S}^{n}\times\mathbb{S}^{n}\times\cdots\mathbb{S}^{n}}\sqrt{\prod_{i=1}^m f^{\lambda^i}}\sqrt{\prod_{i=1}^mf^{\mu^i}}\\
\leq
\int_{\mathbb{S}^{n}\times\mathbb{S}^{n}\times\cdots\mathbb{S}^{n}}\frac{\prod_{i=1}^m f^{\lambda^i}+\prod_{i=1}^mf^{\mu^i}}{2}
=\frac{F(\lambda)}{2}+\frac{F(\mu)}{2}.
\end{equation*}
In case where $\lambda,\mu\in I^m$ are integer partitions the condition $\lambda\sqsubseteq\mu$ is equivalent to $\lambda^{|\mu|}\preceq\mu^{|\lambda|}$ (cf. \cite{cuttler}). Hence applying the Schur convexity of $F$ on $I^{m|\mu}$ we obtain the following.
\begin{align*}
\prod_{i=1}^m\int_{\mathbb{S}^{n}}f^{\lambda_i}&=\Big[\prod_{i=1}^m\int_{\mathbb{S}^{n}}f^{\lambda_i}\Big]^{\frac{|\mu|}{|\mu|}}\\
&=\Big[\underbrace{\int_{\mathbb{S}^{n}}f^{\lambda_1}\int_{\mathbb{S}^{n}}f^{\lambda_1}\cdots\int_{\mathbb{S}^{n}}f^{\lambda_1}}_{|\mu|}\underbrace{\int_{\mathbb{S}^{n}}f^{\lambda_2}\int_{\mathbb{S}^{n}}f^{\lambda_2}\cdots\int_{\mathbb{S}^{n}}f^{\lambda_2}}_{|\mu|}\cdots\underbrace{\int_{\mathbb{S}^{n}}f^{\lambda_m}\int_{\mathbb{S}^{n}}f^{\lambda_m}\cdots\int_{\mathbb{S}^{n}}f^{\lambda_m}}_{|\mu|}\Big]^{\frac{1}{|\mu|}}\\
&\leq\Big[\underbrace{\int_{\mathbb{S}^{n}}f^{\mu_1}\int_{\mathbb{S}^{n}}f^{\mu_1}\cdots\int_{\mathbb{S}^{n}}f^{\mu_1}}_{|\lambda|}\underbrace{\int_{\mathbb{S}^{n}}f^{\mu_2}\int_{\mathbb{S}^{n}}f^{\mu_2}\cdots\int_{\mathbb{S}^{n}}f^{\mu_2}}_{|\lambda|}\cdots\underbrace{\int_{\mathbb{S}^{n}}f^{\mu_m}\int_{\mathbb{S}^{n}}f^{\mu_m}\cdots\int_{\mathbb{S}^{n}}f^{\mu_m}}_{|\lambda|}\Big]^{\frac{1}{|\mu|}}\\
&=\Big[\prod_{i=1}^m\int_{\mathbb{S}^{n}}f^{\mu_i}\Big]^{\frac{|\lambda|}{|\mu|}}.
\end{align*}
\end{proof}
The following  corollaries are the main results of this section.
Applying the apreceding theorem to the case $f(\xi)=<X\xi,\xi>$, whenever $x\in\mathbb{R}^n_+$ we obtain the following generalization of Theorem 4.1 for the case of term-normalized complete homogeneous polynomials to the term-normalized $\lambda$-function as follows.
\begin{corollary}
Let $\lambda\in\mathbb{R}^m$ and $\mu\in\mathbb{R}^l$. If $\lambda\preceq\mu$ then
$$H_\lambda(x)\leq H_\mu(x), \quad x\in\mathbb{R}^n_+.$$
Moreover, if $\lambda,\mu$ are integer partitions  with $\lambda\sqsubseteq\mu$ then
$$\mathfrak{H}_\lambda(x)\leq\mathfrak{H}_\mu(x), \quad x\in\mathbb{R}^n_+.$$
\end{corollary}
Now let $x\in\mathbb{R}^n$ and consider the map $f(\xi)=\big(<X\xi,\xi>\big)^2$ then for any $p\geq0$ we have $$\int_{\mathbb{S}^{n}}f^p(\xi)d\sigma(\xi)<\infty.$$
Notice that if 	$\lambda=(2\nu_1,\cdots,2\nu_m)$ and $\mu=(2\eta_1,2\eta_2,\cdots,2\eta_l)$ are two integer partitions then $\lambda\sqsubseteq\mu$ if and only if $\nu\sqsubseteq\eta$. Applying Theorem \ref{tmm1} to $f(\xi)=\big(<X\xi,\xi>\big)^2$ we obtain
	\begin{corollary}
Let $\lambda\in\mathbb{R}^m_+$ and $\mu\in\mathbb{R}^l_+$. If $\lambda\preceq\mu$ then
$$\prod_{i=1}^m\int_{\mathbb{S}^{n}}\bigg[\big(<X\xi,\xi>\big)^2\bigg]^{\lambda_i}d\sigma(\xi)\leq \prod_{i=1}^l\int_{\mathbb{S}^{n}}\bigg[\big(<X\xi,\xi>\big)^2\bigg]^{\mu_i}d\sigma(\xi), \quad x\in\mathbb{R}^n.$$	
If in addition 	$\nu,\eta$ are integer partitions  with $\nu\sqsubseteq\eta$ then
\begin{equation}\label{fe1}
\mathfrak{H}_{2\nu}(x)\leq\mathfrak{H}_{2\eta}(x), \quad x\in\mathbb{R}^n.
\end{equation}
\end{corollary}	

Note that applying the above equation to the cases

$\begin{cases}
\nu=k-1 \quad \mbox{and} \quad \eta=k\\
\nu=(k-1,k+1) \quad \mbox{and} \quad  \eta=(k,k)
\end{cases}$

provides us with an entire sequence of inequalities similar to Newton's identities which is valid for all $x\in\mathbb{R}^n$. More precisely, we have
\begin{equation}\label{fe2}
\sqrt[2k-2]{H_{2k-2}(x)}\leq\sqrt[2k]{H_{2k}(x)}, \quad x\in\mathbb{R}^n
\end{equation}
and
\begin{equation}\label{fe3}
H_{2k-2}(x)H_{2k+2}(x)\leq H^2_{2k}(x), \quad x\in\mathbb{R}^n.
\end{equation}
Finally,we conclude by noting  that (\ref{fe2}) provides a proof for the open problem concerning  the sequence of inequalities that was raised in \cite{roventa}.

	\bibliographystyle{plainnat}
		
\end{document}